\documentclass[reqno]{amsart}
\newcommand{\ben}{\begin{enumerate}}
\newcommand{\een}{\end{enumerate}}
\newcommand{\bqu}{\begin{quote}}
\newcommand{\equ}{\end{quote}}
\newcommand{\beq}{\begin{equation}}
\newcommand{\eeq}{\end{equation}}
\newcommand{\bec}{\begin{center}}
\newcommand{\ece}{\end{center}}

\newcommand{\ra}{\rightarrow}

\newcommand{\sig}{\sigma}
\newcommand{\Sig}{\Sigma}
\newcommand{\eps}{\epsilon}

\newcommand{\Del}{\Delta}
\newcommand{\lam}{\lambda}
\newcommand{\ah}{\alpha}
\newcommand{\gam}{\gamma}
\newcommand{\Gam}{\Gamma}
\usepackage{url}
\usepackage{latexsym,epsfig,amssymb,amsmath,amsthm,color,url}
\usepackage{hyperref}

\allowdisplaybreaks 
 \numberwithin{equation}{section}
\newtheorem{theorem}{Theorem}[section]

\newtheorem{remark}[theorem]{Remark}
\newtheorem{lemma}[theorem]{Lemma}

\newtheorem{example}[theorem]{Example}

\theoremstyle{definition}
\newtheorem{defn}[theorem]{Definition}

\theoremstyle{remark}

\begin{titlepage}

\title[Short title]{Galton-Watson Probability Contraction
\vspace{0.2in}\\
\begin{small}Joel Spencer, \\
Courant Institute of Mathematical Sciences\\
New York University\\
\vspace{0.2in}
Moumanti Podder, \\
Courant Institute of Mathematical Sciences\\
New York University
\end{small}}

\end{titlepage}

\begin{document}
\bibliographystyle{plainnat}

\maketitle

\begin{center}
\section{Abstract}
\emph{We are concerned with exploring the probabilities of first order statements for Galton-Watson trees with $Poisson(c)$ offspring distribution. Fixing a positive integer $k$, we exploit the $k$-move Ehrenfeucht game on rooted trees for this purpose. Let $\Sigma$, indexed by $1 \leq j \leq m$, denote the finite set of equivalence classes arising out of this game, and $D$ the set of all probability distributions over $\Sigma$. Let $x_{j}(c)$ denote the true probability of the class $j \in \Sigma$ under $Poisson(c)$ regime, and $\vec{x}(c)$ the true probability vector over all the equivalence classes. Then we are able to define a natural \emph{recursion function} $\Gamma$, and a map $\Psi = \Psi_{c}: D \rightarrow D$ such that $\vec{x}(c)$ is a \emph{fixed point} of $\Psi_{c}$, and starting with any distribution $\vec{x} \in D$, we converge to this fixed point via $\Psi$ because it is a contraction. We show this both for $c \leq 1$ and $c > 1$, though the techniques for these two ranges are quite different. } 
\end{center}


\vspace{0.3in}
\section{Introduction}

The Galton-Watson tree (henceforth, GW tree) $T=T_c$ with parameter $c > 0$ is a much studied
object.  It is a random rooted tree.  Each node, independently, has $Z$ children where $Z$
has the Poisson distribution with mean $c$. We let $P_{c}$ denote the probability under $T_{c}$. 

We shall examine the first order language on rooted trees.  This consists of a constant
symbol $R$ (the root), equality $v=w$,  and a parent function $\pi[v]$ defined for all 
vertices $v\neq R$.  (Purists may prefer a binary relation $\pi^*[v,w]$, that $w$ is the parent of $w$.)
Sentences must be finite and made up of the usual Boolean connectives ($\neg,\vee,\wedge,\ldots$)
and existential $\exists_v$ and universal $\forall_v$ quantification over vertices.  The quantifier
depth of a sentence $A$ is the depth of the nesting of the existential and universal quantifiers.

\begin{example} \label{no node with one child}
No node has precisely one child.
\beq \label{no one child node}
\neg[\exists_{u,x} \pi[x]=u  \wedge \forall_z [(z\neq x)\Rightarrow \neg \pi[z]=u]].
\eeq
\end{example}

We outline some of our results.
For any first order sentence $A$ set
\beq\label{fsubA} f_A(c) = \Pr[T_c\models A], \eeq
the probabiity that $T=T_c$ has property $A$.  Except in
examples we will work with the {\em quantifier depth} $k$
of $A$.  The value $k$ shall be arbitrary but {\em fixed}
throughout this presentation.
With (\ref{d}) below we decompose any $f_A(c)$
into its ``atomic" $x_j(c)$.  In Section \ref{sectfixedpoint}
we show that the $x_j(c)$ are solutions to a finite system of
equations involving polynomials and exponentials. The
solution is described as the fixed point of a map $\Psi_c$
over the space of distributions $D$ defined by (\ref{b}). In
Theorem \ref{uniquefixed} we show that this system
has a unique solution.  
In Sections \ref{sectsubcrit} (for the
subcritical case) and \ref{sectproofcontraction} (for the
general case) we show that $\Psi_c$ is a contraction.
Employing the Implicit Function Theorem in Section \ref{sectimplicit}
we then achieve one of our main results:

\begin{theorem}\label{a}  Let $A$ be first order.  Then
$f_A(c)$ is a $C^{\infty}(0,\infty)$ function.  
That is, all derivatives of $f_A(c)$
exist and are continuous at all $c > 0$.
\end{theorem}

\begin{remark} 
Let $y=g(c)$ be the probability that $T=T_c$ is infinite.
It is well known that $g(c)=0$ for $c\leq 1$ while for $c > 1$, $y=g(c)$ is
the unique positive real satisfying $e^{-cy}=1-y$.  The value $c=1$ is often
referred to as a critical, or percolation, point for GW-trees.  The function
$g(c)$ is {\em not} differentiable at $c=1$.  The right sided derivative
$\lim_{c\ra 1^+} (g(c)-g(1))/(c-1)$ is $2$ while the left sided derivative
is zero.  An interpretation of Theorem \ref{a} that we favor is that
the critical point $c=1$ cannot be seen through a First Order lens.
Theorem \ref{a} thus yields that the property of $T$ being infinite is
not expressible in the first language -- though this can be shown with
a much weaker hammer!

\begin{center}
            \includegraphics[width=14cm]{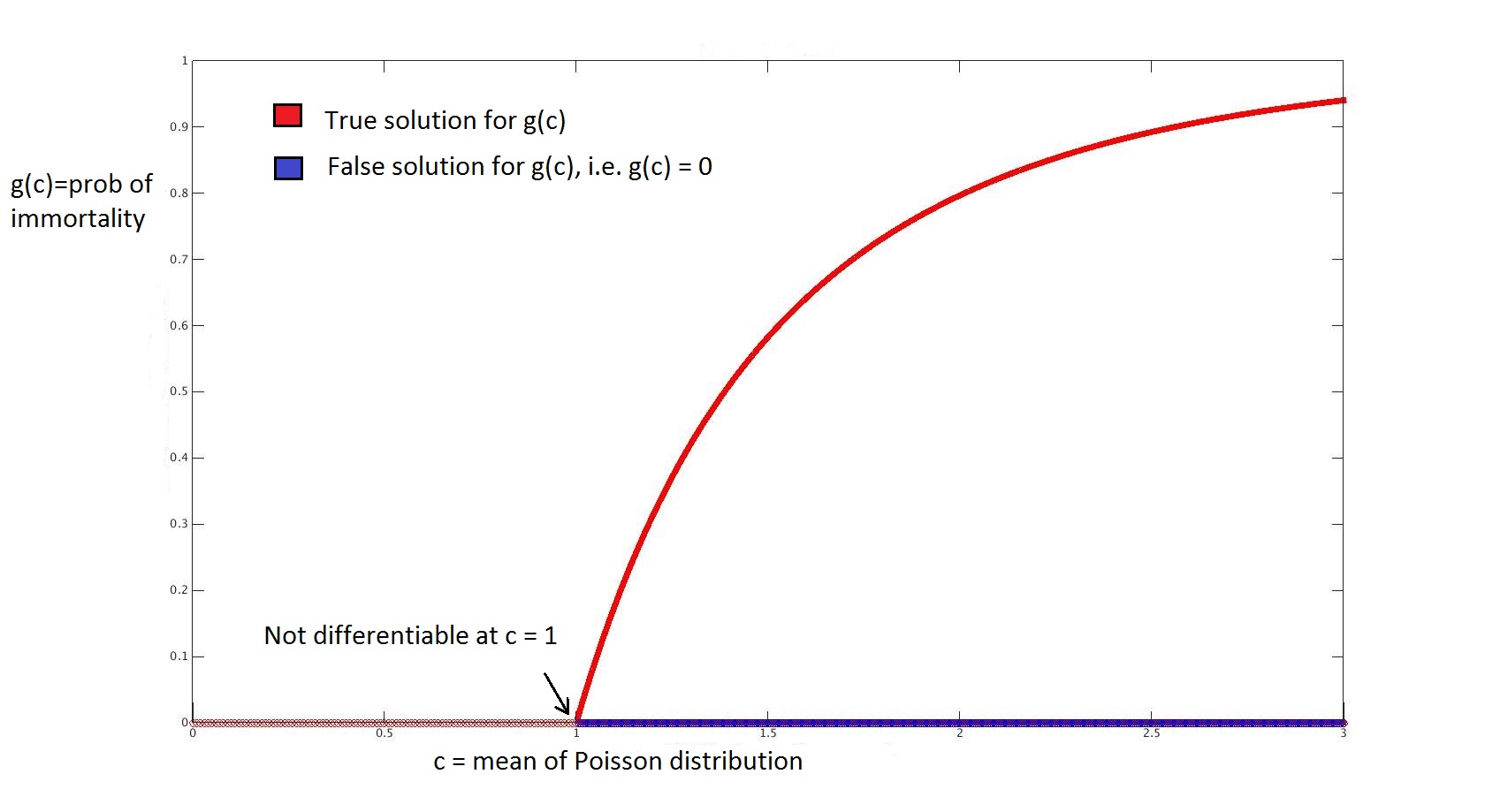}
        \end{center}
The plot clearly shows how the function is not differentiable at $c = 1$, and how the solution is non-unique for $c > 1$.
\par But the plot corresponding to the property in Example \ref{no node with one child} shows that the probability is a smooth function of $c$, which is in keeping with Theorem \ref{a}.

\begin{center}
            \includegraphics[width=13cm]{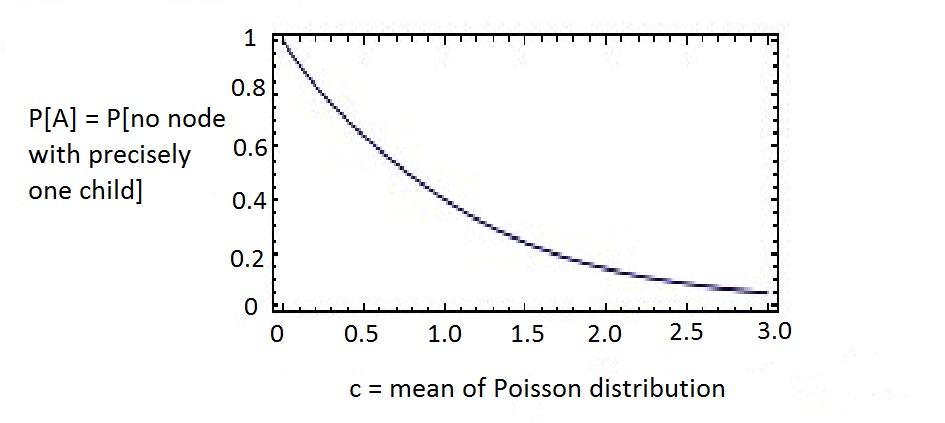}
        \end{center}
\end{remark}

\begin{defn}\label{deftree}
With $v\in T$, $T(v)$ denotes the rooted
tree consisting of $v$ and all of its descendents, with $v$ regarded
as the root.  For $s$ a nonnegative integer, $T|_{s}$ denotes the rooted tree consisting of its vertices up to generation at most $s$.  We call $T|_{s}$ the $s$-cutoff of $T$. (This is defined even if no vertices are at generation $s$.) $T(v)|_{s}$ denotes the $s$-cutoff of $T(v)$.
\end{defn}

\section{The Ehrenfeucht Game}
\subsection{Equivalence Classes}

Let $k$ denote an arbitrary positive integer. 
We may then define an equivalence relation $\equiv_k$
(we often omit the subscript) on all trees $T$, as follows.
\begin{defn}\label{equivk}
$T\equiv_k T'$ if they have the same truth value for
all $A$ of quantifier depth at most $k$.  Equivalently
$T\equiv_k T'$ if Duplicator wins the $k$-move Ehrenfeucht
game $EHR[T,T',k]$.  $\Sig=\Sig_k$ denotes the set
of all equivalence classes.  
\end{defn}

Critically, $\Sig_k$ is a
{\em finite} set.  As a function of $k$ we note that
$|\Sig_k|$ grows like a tower function.  We give
\cite{01} as a general reference to these basic results.

\begin{defn}\label{defev}
For any rooted tree $T$, the Ehrenfeucht value of $T$,
denoted by $EV[T]$, is that equivalence class $\sigma \in \Sigma$ to which $T$ belongs.
\end{defn}

For convenience we denote
the elements of $\Sig$ by $\Sig=\{1,\ldots,m\}$.
We let $D\subset \mathbb{R}^m$ denote the set of all possible probability distributions
over $\Sig$.  That is,
\beq\label{b} D = \{(x_1,\ldots,x_m): \sum_{j=1}^{m} x_j=1 \text{ and all } x_j\geq 0 \}.
\eeq
(We let $\vec{x}(c)$ denote the probability distribution for the equivalence class of $T=T_c$.) Recall that $P_{c}$ denotes the probability under $T_{c}$, the GW process with $Poisson(c)$ offspring. Let 
\begin{equation} \label{c}
x_{j}(c) = P_{c}(j) = P[EV[T_{c}] = j], \quad j \in \Sigma.
\end{equation}
Then $\vec{x}(c) = \{x_{j}(c): 1 \leq j \leq m\}$ denotes the probability vector in $D$ under $P_{c}$.


\begin{theorem}\label{e} 
$\vec{x}(c)$ has derivatives of all orders.  In particular,
each $x_i(c)$ has derivatives of all orders.
\end{theorem}

The proof of Theorem \ref{e} is a goal of this paper, accomplished
only in Section \ref{sectimplicit} after many preliminaries.
Any first order sentence of quantifier depth $A$ is determined, tautologically,
by the set $S(A)$ of those $j\in\Sig$ such that all $T$ with $EV[T]=j$
have property $A$.  For any $j\in\Sig$ either all $T$ with $EV[T]=j$ or
no $T$ with $EV[T]=j$ have property $A$.  We may therefore decompose the $f_A(c)$
of (\ref{fsubA}) into 
\beq\label{d} f_A(c) = \sum_{j\in S(A)} x_j(c).  \eeq
Theorem \ref{a} will therefore follow from Theorem \ref{e}.

\subsection{Recursive States}

In the $k$-move Ehrenfeucht game,  values $\geq k$
are roughly all ``the same."  This will be made precise in the subsequent Theorem \ref{g}. We define
\beq\label{f}  C = \{0,1,\ldots,k-1,\omega\}
\eeq
The phrase ``there are $\omega$ copies" is to be interpreted
as ``there are $\geq k$ copies."  
We call $v\in T$ a {\em rootchild} if its parent is the root $R$.
For $w\neq R$ we say $v$ is the {\em rootancestor} of $w$ 
if $v$ is that unique rootchild with $w\in T(v)$. Of course, a rootchild is its own rootancestor.  

Theorem \ref{g} roughly states  that the Ehrenfeucht value
of a tree $T$ is determined by the Ehrenfeucht values $EV[T(v)]$
for all the rootchildren $v$. To clarify: $\omega$ rootchildren 
means at least $k$ rootchildren while $n$ rootchildren, $n\in C$,
$n\neq\omega$ means precisely $n$ rootchildren.

\begin{theorem}\label{g} 
Let $\vec{n}=(n_1,\ldots,n_m)$ with all $n_j\in C$.
Let $T$ have the property that for all $1\leq j\leq m$ there are
$n_j$ rootchildren $v$ with $EV[T(v)]=j$.
Then $\sig=EV[T]$ is uniquely determined.
\end{theorem}

\begin{defn}\label{def1}  Let
\beq\label{defrec}
\Gam: \{(n_1,\ldots,n_m):n_i\in \{0,1,\ldots,k-1,\omega\} \}\ra \Sig
\eeq
be given by $\sig=\Gam(\vec{n})$ with $\vec{n},\sig$ satisfying
the conditions of Theorem \ref{g}.  Then $\Gam$ is called the
{\em recursion function}.
\end{defn}

\begin{proof}[Proof of Theorem \ref{g}]
 Let $T,T'$ have the same $\vec{n}$.  We give a strategy
for Duplicator in the Ehrenfeucht game $EHR[T,T',k]$.  Duplicator will
create a partial matching between the rootchildren $v\in T$
and the rootchildren $v'\in T'$.  When $v,v'$ are matched, $EV[T(v)]=EV[T'(v')]$.
At the end
of any round of the game call a rootchild $v\in T$ (similarly $v'\in T'$) 
{\em free} if no $w\in T(v)$ has yet been selected.

Suppose Spoiler plays $w\in T$ (similarly $w'\in T'$) with rootancestor $v$.
Suppose $v$ is free.  
Duplicator finds a free $v'\in T'$ with $EV(T(v))=EV(T'(v'))$.
(As the number for each class is the same for $T,T'$ this may be done
when $n_j\neq \omega$.)
When $EV[T(v)] = j \in \Sigma$ and $n_{j} \neq \omega$, then as the number of rootchildren of $T$ with Ehrenfeucht value $j$ is exactly the same as that in $T'$, hence this can be done. In the special case where $n_j=\omega$ the vertex $v'$ may be found as there
have been at most $k-1$ moves prior to this move and so there are at most $k-1$ rootchildren $v'$ with $EV[T(v')] = j$ that are
not free.  Duplicator then matches $v, v'$.  Duplicator can win $EHR(T(v),T(v'),k)$
as $EV[T(v)]=EV[T'(v')]$.  He employs that strategy to find
a response $w'\in T(v')$ corresponding to $w \in T(v)$.  Once $v,v'$ have been matched any move  $z\in T(v)$ 
(similarly $z'\in T(v')$) is responded to with a move in $z'\in T(v')$ using
the strategy for $EHR(T(v),T(v'),k)$.
\end{proof}

\begin{remark}  
Tree automata consist of a finite state space $\Sig$, an integer $k\geq 1$,
a map $\Gam$ as in (\ref{defrec}) and a notion of accepted states.  While first order
sentences yield tree automata, the notion of tree automata is broader.  Tree automata
roughly correspond to second order monadic sentences, a topic we hope to explore
in future work.
\end{remark}

\subsection{Solution as Fixed Point}\label{sectfixedpoint}

We come now to the central idea.
We define, for $c > 0$, a map $\Psi_c: D\ra D$.  Let $\vec{x}=(x_1,\ldots,x_m)\in D$,
a probability distribution over $\Sig$.  Imagine root $R$ has Poisson mean $c$ children.  
To each child we assign, independently, a $j\in\Sig$ with distribution $\vec{x}$.
Let $n_j \in C$ be the number of children assigned $j$.  Let
$\vec{n}=(n_1,\ldots,n_m)$.  Apply the recursion function (equation \ref{defrec}) $\Gamma$ to get
$\sig=\Gam(\vec{n})$, the Ehrenfeucht value of the root $R$.  Then $\Psi_c(\vec{x}) \in D$
is the distribution thus induced on the Ehrenfeucht value of $R$.

\par The special nature of the Poisson distribution allows a concise expression. When the initial distribution is $\vec{x}$,
the number of chilren assigned $j$ will have a Poisson distribution with
mean $cx_j$ {\em and} these numbers are mutually independent over $j\in\Sig$.
Thus
\beq\label{h1}
\Pr[n_j=u] = e^{-cx_j}\frac{(cx_j)^u}{u!} \quad \text{for } u \in C, u\neq \omega,
\eeq
and 
\beq\label{h2}
\Pr[n_j=\omega] = 1 - \sum_{u=0}^{k-1}\Pr[n_j=u].
\eeq
From the independence, for any $\vec{a}=(a_1,\ldots,a_m)$
with $a_1,\ldots,a_m\in C$,
\beq\label{h3}
\Pr[\vec{n}=\vec{a}] = \prod_{j=1}^m \Pr[n_j=a_j]. 
\eeq
Thus, writing $\Psi_c(x_1,\ldots,x_m)=(y_1,\ldots,y_m)$,
\beq\label{h4}
y_j = \Sig \Pr[\vec{n}=\vec{a}]
\eeq
where the summation is over all $\vec{a}$ with $\Gamma(\vec{a})=j$.

We place all $\Psi_c$ into a single map $\Del$:
\beq\label{i} \Del: D\times (0,\infty)\ra D \mbox{ by } \Del(\vec{x},c)=\Psi_c(\vec{x}).
\eeq
Setting $\Del(x_1,\ldots,x_m,c)=(y_1,\ldots,y_m)$, the $y_j$ are finite sums of products
of polynomials and exponentials in the variables $x_1,\ldots,x_m,c$.  In particular,
all partial derivatives of all orders exist everywhere.

Recall (\ref{c}),  $\vec{x}(c)$ denotes the probability distribution for the equivalence classes under the probability measure $P_{c}$ $T=T_c$.  

\begin{theorem}\label{j} $\vec{x}(c)$ is a fixed point for $\Psi_c:D\ra D$.  That is,
$\Psi_c(\vec{x}(c)) = \vec{x}(c)$.  
\end{theorem}


\begin{proof}
To show that $\vec{x}(c)$ is a fixed point of iteration $\Psi_{c}$, we start with the initial probability vector $\vec{x} = \vec{x}(c)$. Once we perform the iteration, using the definition of the recursive function $\Gamma$ from \eqref{defrec}, we know, for any $j \in \Sigma$, 
\begin{align}
\Psi_{c}(\vec{x}(c))(j) =& \sum_{\vec{a}: \Gamma(\vec{a}) = j} \prod_{i=1}^{m} P[n_{i} = a_{i}| \vec{x} = \vec{x}(c)], \quad \text{from } \eqref{h3}, \nonumber\\
=& \sum_{\vec{a}: \Gamma(\vec{a}) = j} \prod_{i=1}^{m} P[Poisson(c \cdot x_{i}(c)) = a_{i}] \nonumber\\
=& \sum_{\vec{a}: \Gamma(\vec{a}) = j} P_{c}[\vec{n} = \vec{a}] \nonumber\\
=& P_{c}[j] = x_{j}(c), \nonumber
\end{align}
where recall that $P_{c}$ is the probability induced under the $Poisson(c)$ offspring distribution.
\end{proof}

\begin{example}  
For many particular $A$ the size of $\Sigma$, which may be 
thought of as the state space, may be reduced considerably.  Let $A$ be
the property given in \eqref{no one child node}, that no node has precisely one child.  We define
state $1$, that $A$ is true and state $2$, that $A$ is false.  We set
$C=\{0,1,\omega\}$ with $\omega$ meaning ``at least two."  Let $n_1,n_2\in C$
be the number of rootchildren $v$ with $T(v)$ having state $1,2$ respectively.
Then $T$ is in state $1$ if and only if $\vec{n}=(n_1,n_2)$ has one of
the values $(0,0),(\omega,0)$.  Let $D$ be the set of distributions on the
two states, $D=\{(x,y):0\leq x\leq 1, y=1-x\}$.  Then $\Psi_c(x,y)=(z,w)$
with $w=1-z$ and
\begin{equation} \label{k1}
z = e^{-cx} e^{-cy} + e^{-cy}[1 - e^{-cx} - (cx)e^{-cx}] = e^{-cy}[1 - (cx) e^{-cx}].
\end{equation}
The fixed point $(x,y)$ then has $x=\Pr[A]$ satisfying the equation
\begin{equation} \label{k2}
x = e^{-c(1-x)}[1 - (cx)e^{-cx}].
\end{equation}

\end{example}

\begin{example}
Let $A$ be that there is a vertex $v$ with 
precisely one child who has precisely one child.  Let state $1$
be that $A$ is true.  Let state $2$ be that $A$ is false but
that the root has precisely one child.  Let state $3$ be all
else.  Set $C=\{0,1,\omega\}$.  Set $D=\{x,y,z:x+y+z=1, x \geq 0, y \geq 0, z \geq 0\}$.
$T$ is in state $1$ if and only if $\vec{n}=(n_1,n_2,n_3)$ has
either $n_1\neq 0$ or $n_1=0,n_2=1,n_3=0$.  $T$ is in state $2$
if and only if $\vec{n}=(0,0,1)$.  Then
$\vec{x}(c)=(x,y,z)$ must satisfy the system (noting $z=1-x-y$ is
dependent)
\beq\label{l1}
x = (1-e^{-cx})+ e^{-cx}(cye^{-cy})e^{-cz} = 1-e^{-cx}+cye^{-c}
\eeq
\beq\label{l2}
y = e^{-cx}e^{-cy}(cze^{-cz})= cze^{-c}
\eeq
Here $x=\Pr[A]$.  In general, however, $\Pr[A]$ will be the
sum (\ref{d}).
\end{example}

\section{The Contraction Formulation}

\subsection{The Total Variation Metric}

On $D$ we let $\rho(\vec{x},\vec{y})$ denote the usual
Euclidean metric, and $||\vec{x} - \vec{y}||_{1}$ the $L^{1}$ distance.  We let $TV(\vec{x},\vec{y})$ denote
the total variation distance. With $\vec{x}=(x_1,\ldots,x_m)$
and $\vec{y}=(y_1,\ldots,y_m)$ this standard metric
is given by
\beq\label{tvdef}
TV(\vec{x},\vec{y}) = \frac{1}{2}|\vec{x}-\vec{y}|_1 =
\frac{1}{2}\sum_{j=1}^m |x_j-y_j|.
\eeq


\par Total variation distance between any two probability distributions $\mu, \nu$ on the same probability space has a natural interpretation in terms of coupling $\mu$ and $\nu$. Let $p = \displaystyle \sum_{x} \mu(x) \wedge \nu(x)$. Flip a coin that has probability of heads $p$.
\begin{itemize}
\item If it lands heads, then choose a value $Z$ according to the probability distribution
\begin{equation}
F_{min}(x) = \frac{\mu(x) \wedge \nu(x)}{p},
\end{equation}
and set $X = Y = Z$.
\item If the coin lands tails, choose $X$ according to the probability distribution 
\[F_{1}(x) = \left\{
  \begin{array}{lr}
    \frac{\mu(x) - \nu(x)}{TV(\mu, \nu)} & : \text{ if } \mu(x) > \nu(x);\\
    0 & : \text{ otherwise}.
  \end{array}
\right.
\]
Independently choose $Y$ according to the probability distribution
\[F_{2}(x) = \left\{
  \begin{array}{lr}
    \frac{\nu(x) - \mu(x)}{TV(\mu, \nu)} & : \text{ if } \nu(x) > \mu(x);\\
    0 & : \text{ otherwise}.
  \end{array}
\right.
\]
\end{itemize}
Then $X \sim \mu, Y \sim \nu$ and $(X, Y)$ are coupled in such a way that
\begin{equation} \label{optimal coupling}
TV(\mu, \nu) = P[X \neq Y].
\end{equation}
We call such a coupling \emph{optimal}. \\

\par As such, we have the general well-known result:
\begin{theorem}
For any two probability distributions $\mu, nu$ on a common probability space, 
\begin{equation} \label{tv inequality}
TV(\mu, \nu) = \inf\{P[X \neq Y]: (X, Y) \text{ any coupling of } \mu, \nu\}.
\end{equation}
\end{theorem}
We refer the reader to \cite{markov chains mixing times} for further reading.

\subsection{The Contraction Theorem}

\begin{theorem}\label{thmcontract} For all $c>0$ there
exists a positive integer $s$ and an $\ah < 1$
such that for all 
$\vec{x},\vec{y}\in D$
\beq\label{m1} \rho(\Psi^s(\vec{x}), \Psi^s(\vec{y}))
\leq \ah \rho(\vec{x},\vec{y}).
\eeq
\end{theorem}

The map $\Psi^s:D\ra D$ has a natural interpretation.  Let
$\vec{x}=(x_1,\ldots,x_m)\in D$.  Generate a random GW
tree $T=T_c$ but stop at generation $s$.  (The root is
at generation $0$.)  To each node (there may not be
any) $v$ at generation $s$ assign independently 
$j\in \Sig$ from distribution $\vec{x}$.  Now we
work up the tree towards the root.  Suppose, formally
by induction, that all $w$ at generation $i$ have
been assigned some $j\in\Sig$.  A $v$ at generation
$i-1$ will then have $n_j$ children assigned $j$
(allowing $n_j=\omega$).  The value at $v$, 
which is now determined by Theorem \ref{g},
is given by the recursion function 
$\Gam(\vec{n})$ of Definition \ref{defrec}.
$\Psi^s(\vec{x})$ will then be the distribution of
the Ehrenfeucht value assigned to the root.

\begin{remark} 
The non-first order property $A$ that $T$
is infinite may be similarly examined.  Set $C=\{0,\omega\}$ ($\omega$ denoting $\geq 1$)
and let state 1 be that $T$ is infinite, state 2 that it is
not.  $T$ is in state 1 if and only if $\vec{n}=(\omega,0)$
or $(\omega,\omega)$.  Then $D=\{(x, 1-x): 0 \leq x \leq 1\}$ and
\beq\label{psiinf}
\Psi_c(x,1-x) = (1-e^{-cx},e^{-cx}).
\eeq
However, $\Psi_c$ has two fixed points: $(0,1)$ and the 
``correct" $(x(c),1-x(c))$ when $c > 1$.  The contraction
property (\ref{m1}) will not hold.  With $\eps$
small, $1-e^{-c\eps}\sim c\eps$ and so $\vec{x}=(0,1)$,
and $\vec{y}=(\eps,1-\eps)$ become further apart
on application of $\Psi_c$, when $c > 1$.
\end{remark}

\subsection{The Subcritical Case}\label{sectsubcrit}  Here we prove
Theorem \ref{thmcontract} under the additional assumption
that $c < 1$.  The proof in this case is considerably
simpler.  Further, it may shed light on the general
proof.

\begin{theorem}\label{n}  For any $c < 1$ and any
$\vec{x},\vec{y}\in D$
\beq\label{n1} TV(\Psi_{c}(\vec{x}),\Psi_{c}(\vec{y})) \leq c\cdot TV(\vec{x},\vec{y}).
\eeq 
\end{theorem}


\begin{proof} 
The main idea is to use suitable coupling of $\vec{x}$ and $\vec{y}$. First we fix $s \in \mathbb{N}$. We then create two pictures. In both pictures, let $v$ have $s$ many children $v_{1}, \ldots v_{s}$. In picture $1$, we assign, mutually independently, labels $X_{1}, \ldots X_{s} \in \Sigma$ to $v_{1}, \ldots v_{s}$ respectively, with $X_{i} \sim \vec{x}, 1 \leq i \leq s$. In picture $2$, we assign, again mutually independently, labels $Z_{1}, \ldots Z_{s} \in \Sigma$ to $v_{1}, \ldots v_{s}$, with $Z_{i} \sim \vec{y}, 1 \leq i \leq s$. The pairs $(X_{i}, Z_{i}), 1 \leq i \leq s$ are mutually independent, but for every $i$, $(X_{i}, Z_{i})$ is coupled so that
\begin{equation} \label{coupling subcritical}
P[X_{i} \neq Z_{i}] = TV(\vec{x}, \vec{y}).
\end{equation}
Suppose $X_{v}$ is the label of the root $v$ in picture $1$ that we get from the recursion function $\Gamma$ (from \eqref{defrec}), and $Z_{v}$ that in picture $2$. Then $X_{v} \sim \Psi_{c}(\vec{x}), \quad Z_{v} \sim \Psi_{c}(\vec{y})$. 
\begin{align}
TV(\Psi_{c}(\vec{x}), \Psi_{c}(\vec{y})) \leq & P[X_{v} \neq Z_{v}] \nonumber\\
\leq & \sum_{s=0}^{\infty} P[Poisson(c) = s] \sum_{i=1}^{s} P[X_{i} \neq Z_{i}] \nonumber\\
=& \sum_{s=0}^{\infty} P[Poisson(c) = s] \cdot s \cdot TV(\vec{x}, \vec{y}) \nonumber\\
=& \lambda \cdot TV(\vec{x}, \vec{y}). \nonumber
\end{align}
\end{proof}

\begin{theorem} 
Theorem \ref{thmcontract} holds when $c < 1$. 
\end{theorem}

\begin{proof}
The inequalities
\beq\label{m4}
|\vec{z}|_1 \geq |\vec{z}|_2 \geq m^{-1/2}|\vec{z}|_1
\eeq
bound the $L^1$ and $L^2$ norms on $\mathbb{R}^m$ by multiples of each other.
As $TV(\vec{x},\vec{y}) = \frac{1}{2}|\vec{x}-\vec{y}|_1$,
\beq\label{m2}
TV(\vec{x},\vec{y}) \geq \frac{1}{2}\rho(\vec{x},\vec{y}) \geq  m^{-1/2}\cdot TV(\vec{x},\vec{y}).
\eeq

\par Applying Theorem \ref{n} repeatedly,
\beq\label{n2} TV(\Psi_{c}^s(\vec{x}),\Psi_{c}^s(\vec{y})) \leq 
c^s\cdot TV(\vec{x},\vec{y}).
\eeq 
Combining (\ref{m2} and \ref{n2})
\beq\label{n3} 
\rho(\Psi_{c}^s(\vec{x}),\Psi_{c}^s(\vec{y})) \leq 2\cdot TV(\Psi_{c}^s(\vec{x}),\Psi_{c}^s(\vec{y})) 
\leq 2c^s\cdot TV(\vec{x},\vec{y}) \leq 
2c^s\sqrt{m}\rho(\vec{x},\vec{y}).
\eeq
We select $s$ so that $2c^s\sqrt{m}<1$ and set $\ah=2c^s\sqrt{m}$.
\end{proof}

\section{Universality}\label{sectuniversality}

We define a function $Rad[i]$ on the nonnegative integers by
the recursion
\beq\label{defrad}
Rad(0) = 0 \mbox{ and } Rad(i+1)=3R(i)+1 \mbox{ for } i\geq 0.
\eeq

\begin{defn}\label{defrho}
In $T$ we define a distance $\rho(v,w)$ to be the minimal $r$ for which
there is a sequence $v=z_0,z_1,\ldots,z_r=w$ where each $z_{i+1}$ is
either the parent {\em or} a child of $z_i$.  We set $\rho(v,v)=0$.
\end{defn}

As an example, cousins would be at distance four.

\begin{defn}\label{defball} For $r$ a nonnegative integer, $v\in T$,
the ball of radius $r$ around $v$, denoted $B(v,r)$ is the set of $w\in T$
with $\rho(v,w)\leq r$.  We consider $v$ a designated vertex of $B(v,r)$.
\end{defn}

We define an equivalence relation, depending on $k$, on such balls.  

\begin{defn}\label{defequiv} 
$B(v,r)\equiv_k B(v',r)$ if the two sets
satisfy the same first order sentences of quantifier depth at most
$k-1$ with $v,v'$ as designated vertices, allowing $\pi$, $=$, and
$\rho$.
\end{defn}

\begin{remark}
Note that the $(k-1)$-round Ehrenfeucht game with $v,v'$ designated
is identical to the $k$-round Ehrenfeucht game in which the first round is
mandated to select $v,v'$.
\end{remark}

\par Equivalently, $B(v,r)\equiv_k B(v',r)$ if Duplicator wins the $k$-move
Ehrenfeucht game on these sets in which the first round is mandated to
be $v,v'$ and Duplicator must preserve $\pi$, $=$ and $\rho$.  The
distance function $\rho$ could be replaced by the 
binary predicates $\rho_i(w_1,w_2): \rho(w_1,w_2)=i$, $0\leq i\leq 2r$.
As this is a {\em finite} number of predicates the number of equivalence
classes is finite.  Let $\Sig_k^{BALL}$ denote the set of equivalence classes.

\begin{defn}\label{defdisj}
We say $S_1,S_2\subset T$ are {\em strongly
disjoint} if there are no $v_1\in S_1$, $v_2\in S_2$ with $\rho(v_1,v_2)\leq 1$.
\end{defn}

\begin{defn}\label{deffull}  
We say $T$ is $k$-{\em full} if for any
$v_1,\ldots,v_{k-1}\in T$ and any $\sig\in\Sig_k^{BALL}$ there exists
a vertex $v$ such that
\ben
\item $B(v,Rad(k-1))$ is in equivalence class $\sig$.
\item $B(v,Rad(k-1))$ is strongly disjoint from all $B(v_i,Rad(k-1))$.
\item $B(v,Rad(k-1))$ is strongly disjoint from $B(R,Rad(k))$, $R$ the root.
\een
\end{defn}

When $T$ is $k$-full our next result shows that the truth value of 
first order sentences of quantifier depth at most $k$ is determined
by examining $T$ ``near" the root.  This ``inside-outside" strategy
is well known, see, for example, \cite{01}.

\begin{theorem}\label{onlytop}  
Let $T,T'$ with roots $R,R'$ both be $k$-full.
Suppose, as per Definition \ref{defequiv}, $B(R,Rad(k))\equiv_{k+1} B(R',Rad(k))$. 
Then $T,T'$ have the same $k$-Ehrenfeucht value as
given by Definition \ref{defev}.
\end{theorem}

\begin{proof}
Let $T,T'$ satisfy the condition of Theorem \ref{onlytop}.
We give a strategy for Duplicator to win the $k$-move Ehrenfeucht game.  For convenience
we add a move zero in which the roots $R,R'$ are selected.
Suppose $i$ moves remain.  Consider the union of the balls of radius
$Rad(i)$ about the chosen vertices (including the root) in each tree.  
These split into components.  Duplicator shall insure that the
corresponding chosen vertices are in the same components and that
the components are equivalent.  At the start, with $i=k$, this is
true by assumption, because $B(R, Rad(k)) \equiv_{k} B(R', Rad(k))$.  Suppose this holds with $i$ moves remaining
and Spoiler selects (the other case being symmetric) $v\in T$.
There are two cases.
\\ {\tt Inside:}  $v$ is at distance at most $2Rad(i-1)+1$ from a
previously selected $v_s$.  Then its ball of radius $Rad(i-1)$ lies
entirely inside (from the recursion (\ref{defrad})) the ball of 
radius $Rad(i)$ around $v_s$. Duplicator then considers the
equivalent component in $T'$ and moves the corresponding $v'$.
\\ {\tt Outside:}  Now the ball of radius $Rad(i-1)$ about $v$
lies in a separate component from the balls of radius $Rad(i-1)$
about the previously chosen vertices.  As $T'$ is $k$-full,
Duplicator selects $v'\in T'$ satisfying the conditions of
Definition \ref{deffull}.
\par In either case Duplicator continues the property.  At
the end of the game there are zero moves left.  The union
of the balls of radius zero, the vertices selected, are 
equivalent in $T,T'$ and Duplicator has won.
\end{proof}

\vspace{0.1in}

\begin{defn}
\textbf{Christmas Tree:} We replace the complex notion of $k$-full
by a simpler sufficient condtion.  For each $\sig\in\Sig_k^{BALL}$
create $k$ copies of a ball in that class.  Take a root vertex
$v$ and on it place $k\cdot |\Sig_k^{BALL}|$ disjoint paths
(parent to child) of length $Rad(k)+Rad(k-1)+1$.  Make each endpoint
the top of one of these copies. 
\end{defn}

\begin{defn}\label{defunivk}
The $k$-universal tree, denoted $UNIV_k$,
is the Christmas Tree defined above.
\end{defn}

\begin{defn}\label{suniv}  
$T$ is called $s$-universal (given a fixed
positive integer $k$) if all $T'$ with $T|_{s} \cong T'|_{s}$ have the same
$k$-Ehrenfeucht value. Thus $EV[T]$ is determined by $T|_{s}$ completely.
\end{defn}

\begin{theorem}\label{xmas}  
If for some $v$, $T(v)\cong UNIV_k$ then $T$
is $k$-full.  Thus, by Theorem \ref{onlytop}, the $k$-Ehrenfeucht value
of $T$ is determined by $T|_{Rad(k)}$, or in other words, $T$ is $Rad(k)$-universal.
\end{theorem}

\begin{remark}  
Many other trees could be used in place of $UNIV_k$, we
use this particular one only for specificity.
\end{remark}

\begin{remark}  A subtree $T(v)$, where $v$ is not the root, cannot determine the Ehrenfeucht value of
$T$ as, for example, it cannnot tell us if the root has, say, precisely 
two children. Containing this universal tree $UNIV_k$ tells us everything
about the Ehrenfecuht value of $T$ {\em except} properties relating to
the local neighborhood of the root.
\end{remark}

\section{Rapidly Determined Properties}

We consider the underlying probability space for the GW tree
$T=T_c$ to be an infinite sequence $X_1,X_2,\ldots$ of independent
variables, each Poisson with mean $c$.  These naturally create a
tree.  Let the root have $X_1$ children.  Now we go through the
nodes in a breadth first manner.  Let the $i$-th node (the root
is the first node) have $X_i$ children.  This creates a unique
rooted tree.  Note, however, that when the tree is finite with,
say, $n$ nodes, then the values $X_j$ with $j > n$ are irrelevant.
In that case we say that the process {\em aborts} at time $n$.

We employ a useful notation of Donald Knuth.

\begin{defn}\label{deffairly}  
We say an event occurs \emph{quite surely} if the probability
that it does not occur drops exponentially in the given parameter.
\end{defn}

\begin{defn}\label{defrapidly}  
Let $A$ be any property or function of rooted trees.  We
say that $A$ is {\em rapidly determined} if quite surely (in $s$, with
$T=T_c$ and $c$ given)
$X_1,\ldots X_s$ tautologically determine $A$.
\end{defn}

\begin{remark}
Consider the property that $T$ is infinite and suppose $c > 1$.
Given $X_1,\ldots,X_s$ if the tree has stopped then we know it is finite.
Suppose however (as holds with positive limiting probability) after $X_1,\ldots X_s$
the tree is continuing.  If at that stage there are many nodes we can be reasonably certain that $T$ will be infinite, but we cannot be tautologically sure.  This
property is {\em not} rapidly determined.
\end{remark}

\begin{remark} 
In this work we restrict the language in which $A$ is expressed.
It has been suggested that another approach would be
to restrict $A$ to rapidly determined properties.
\end{remark}  

\begin{theorem}\label{findfinite}  Let $T_0$ be an arbitrary finite tree.  Let $A$ be the
(non first order) property that {\em either} the process has aborted by
time $s$ or there
exists $v\in T$ with $T(v)\cong T_0$.  Then $A$ is rapidly determined in parameter $s$.
\end{theorem}

The proof is given in \cite{MPJS}.  Let $T_0$ have depth $d$.
Roughly speaking, when we examine $X_1,\ldots,X_s$ either the process
has aborted or it has not.  If not, quite surely some $i\leq s\eps$
has $T(i)\cong T_0$.  Here $\eps$ is chosen small enough (dependent on $c,d$) so that 
quite surely the children of all $i\leq s\eps$ down $d$ generations. 

\begin{theorem}\label{rapidfirst}  Every first order property $A$ is rapidly
determined.
\end{theorem}

\begin{proof}
Let $A$ have quantifier depth $k$.  Let $T_0$ be the universal
tree $UNIV_k$ as given by Theorem \ref{xmas}.  From Theorem \ref{findfinite}
if $T$ has not aborted by time $s$ then quite surely some $T(i)\cong T_0$.
But then $T$ is already $k$-full and already has depth at least $Rad(k)$.
By Theorem \ref{onlytop} the $k$-Ehrenfeucht value of $T$, hence the
truth value of $A$, is determined solely by $T|_{Rad(k)}$, and hence tautologically by $X_{1}, \ldots X_{s}$.
\end{proof}

\begin{theorem}\label{rapidsuniv}  Fix a positive integer $k$.  Let
$T\sim T_c$.  Then quite surely (in $s$), $T$ is $s$-universal.
\end{theorem}

Theorem \ref{rapidfirst} gives that the $k$-Ehrenfeucht value of $T$
is quite surely determined by $X_1,\ldots,X_s$.  When this is so
it is tautologically determined by $T|_{s}$, which has more information.

\section{Unique Fixed Point}\label{sectuniquefixed}

\begin{theorem}\label{uniquefixed}  
The map $\Psi_c:D \ra D$ has
a {\em unique} fixed point.  
\end{theorem}

\begin{proof}
Let $f(s)$ be the probability that $T_c$ is {\em not}
$s$-universal. For any $\vec{y},\vec{z}\in D$ we couple $\Psi_{c}^s(\vec{y}),
\Psi_{c}^s(\vec{z})$. Create $T_c$ down to generation $s$ and then give
each node at generation $s$ a $\sig\in \Sig$ with independent 
distribution $\vec{y}$, respectively $\vec{z}$.  Then $\Psi_{c}^s(\vec{y}),
\Psi_{c}^s(\vec{z})$ will be the induced state of the root.  But when
$T_c$ is $s$-universal this will be the same for any $\vec{y},\vec{z}$.
Hence $TV[\Psi_{c}^s(\vec{y}),\Psi_{c}^s(\vec{z})] \leq f(s)$.  When $\vec{y},\vec{z}$
are fixed points of $\Psi$, $TV[\vec{y},\vec{z}]\leq f(s)]$.  As
$f(s)\ra 0$, $\vec{y}=\vec{z}$.
\end{proof}

\begin{remark}
Theorem \ref{uniquefixed} will also follow from the more powerful Theorem \ref{thmcontract}.
\end{remark}

\begin{remark} 
It is a challenging exercise to show directly that the
solution $x$ to (\ref{k2}) or the solution $x,y$ to the 
system (\ref{l1} and \ref{l2}) are unique.
\end{remark}

\section{A Proof of Contraction}\label{sectproofcontraction}
\subsection{A Two Stage Process}\label{secttwostage}
Here we prove Theorem \ref{thmcontract} for arbitrary $c$.  Let $D_{0}$ be the depth of $UNIV_k$, as given by Definition \ref{defunivk}.
We shall set
\beq\label{sets}
s=s_0+D_{0}
\mbox{ with } s_0\geq   2\cdot Rad(k)
\eeq
and think of $T|_{s}$ as being generated in two stages.
In Stage 1 we generate $T|_{s_0}$.  From Theorem \ref{rapidsuniv}, by taking
$s_0$ large,  this will be $s_0$-universal with probability near one.
In Stage 2 we begin with an arbitrary but fixed $T_0$ of depth at most $s_0$.  
(We say ``at most" because it includes the possibility that $T_0$ has
no vertices at depth $s_0$.)
From each
node at depth $s_0$, mutually independently, we generate a GW-tree down to depth $D_{0}$.  We
denote by $Ext(T_0)$ this random tree, now of depth (at most) $s$.

\begin{defn}\label{defbad}  
For any $T_0$ of depth at most $s_0$,
$BAD[T_0]$ is the event that  $Ext(T_0)$ is {\em not} $s_{0}$-universal. 
\end{defn}

\begin{theorem}\label{thmuniv}  There exists positive $\beta$ such that for
any $T_0$ of depth at most $s_0$
\beq\label{probbad}  \Pr(BAD[T_0]) \leq e^{-t\beta}  \eeq 
where $t$ denotes the number of nodes of $T_0$ at depth $s_0$.
\end{theorem}

\begin{proof} 
Let $v_1,\ldots,v_t$  denote the nodes of $T_0$ at generation $s_0$.
Each of them independently generates a GW tree.  Let $1-e^{-\beta}$
denote the probability that $T(v_i)\cong UNIV_k$.  With probability
$e^{-t\beta}$ no $T(v_i)\cong UNIV_k$.  But otherwise $Ext(T_0)$ is
$s_{0}$-universal.
\end{proof}

\subsection{Splitting the Extension}

Let $T_0$ be an arbitrary tree of depth $s_0$. Let $\vec{x}\in D$.  Assign to the depth $s$ nodes of $Ext(T_0)$ independent
identically distributed labels $j \in \Sig$ taken from distribution $\vec{x}$.
Applying the recursion function $\Gamma$ of Definition \ref{def1} repeatedly up the generations yields a unique Ehrenfeucht value for the root $R$.  

\begin{defn}\label{defpsiT0}
$\Psi_{c}^s(T_0,\vec{x})$ denotes the induced distribution of the Ehrenfeucht value for
the root $R$ as derived in the description above.
\end{defn}

\begin{theorem}\label{thmtwostage}
\beq\label{aaa} TV(\Psi_{c}^s(\vec{x}),\Psi_{c}^s(\vec{y})) 
\leq 
\sum \Pr(T|_{s_{0}} = T_0)\cdot TV(\Psi_{c}^s(T_0,\vec{x}),
\Psi_{c}^s(T_0,\vec{y}))
\eeq
where the sum is over all $T_0$ of depth (at most) $s_0$.
\end{theorem}

\begin{proof}
We split the distribution of $T$ into the distribution of $Ext(T_0)$,
with probability $\Pr[T|_{s_{0}}=T_0]$, over each $T_0$ of depth (at most) $s_0$.
\end{proof}

\subsection{Some Technical Lemmas}

Let $X=X(c,s)$ be the number of 
children at generation $s$ of the GW tree $T=T_c$.  Let $Y$ be the
sum of $t$ independent copies of $X$.  The next result (not the best
possible) is that the tail of $Y$ is bounded by exponential decay in 
$t$.

\begin{lemma}\label{lem1}  There exists $\beta > 0$ and $y_0$ such that for $y\geq y_0$
\beq\label{cher1}
\Pr[Y \geq yt] \leq e^{-yt\beta}.
\eeq
\end{lemma}

\begin{proof}
Set $f(\lam)=\ln[E[e^{X\lam}]]$.  We employ Chernoff bounds
suboptimally, taking simply $\lam = 1$.  (We require here a standard argument
that $E[e^X]$ is finite.)
Then $E[e^Y]= e^{t\cdot f(1)}$  and
\beq\label{cher2}
\Pr[Y\geq yt] \leq E[e^Y]e^{-yt} \leq e^{(f(1)-y)t}.  \eeq
For $y\geq 2f(1)$, $f(1)-y \leq -y/2$ and we may take $\beta=\frac{1}{2}$.
\end{proof}

\begin{lemma}\label{lem2}  Let $K,\gam > 0$.  Let $BAD$ be an event with
$Pr[BAD] \leq Ke^{-t\gam}$.  Then, for a positive constants $k,\kappa$,
\beq\label{cher3}     
E[Y \ \mathbf{1}_{BAD}] \leq  kte^{-t\kappa},  \eeq
where $Y$ is as defined above.
\end{lemma}

\begin{remark}  
In the worst case the event $BAD$ would coincide with the
top probability $Ke^{-t\gam}$ in the distribution of $Y$.
\end{remark}

\begin{proof}[Proof of Lemma \ref{lem2}]
 



We split $Y$ into $Y < y_{1}t$ and $Y \geq y_{1}t$, where $y_{1}$ needs to be chosen suitably. 
\begin{equation} \label{split}
E[Y \ \mathbf{1}_{BAD}] = E[Y \ \mathbf{1}_{BAD} \ \mathbf{1}_{Y < y_{1}t}] + E[Y \ \mathbf{1}_{BAD} \ \mathbf{1}_{Y \geq y_{1}t}],
\end{equation}
where the first term is bounded by $y_{1} t P[BAD] \leq K y_{1} t e^{-t\gam}$. The second term is bounded above by $E[Y \ \mathbf{1}_{Y \geq y_{1}t}]$, which we use Chernoff type argument to bound. First, recall that $Y$ is the sum of $t$ i.i.d.\ copies of $X = X(c, s)$ which is the number of nodes at generation $s$ of $T_{c}$. Suppose $\varphi_{s}$ denotes the cumulant-generating function of $X$, defined as
\begin{equation} \label{cumulant}
\varphi_{s}(\lambda) = \log E[e^{\lambda X}]. 
\end{equation}
\par We fix some $\lambda > 1$ and choose 
\begin{equation} \label{suitable y1}
y_{1} = \max\left\{y_{0}, \frac{\gamma + \varphi_{s}(\lambda)}{\lambda}\right\},
\end{equation}
where $\gamma$ is as in the bound of $P[BAD]$. From \eqref{cher1} and \eqref{suitable y1}, we then have
\begin{align}
E[Y \ \mathbf{1}_{Y \geq y_{1}t}] \leq & y_{1} t P[Y \geq y_{1}t] + \sum_{k = \lfloor y_{1} t \rfloor + 1}^{\infty} P[Y \geq k] \nonumber\\
\leq & y_{1} t e^{-y_{1} t \beta} + E\left [e^{\lambda Y}\right] \int_{y_{1} t}^{\infty} e^{-\lambda x} dx \nonumber\\
=&  y_{1} t e^{-y_{1} t \beta} + e^{\varphi_{s}(\lambda) t} \frac{1}{\lambda} e^{-\lambda y_{1} t} \nonumber\\
\leq & y_{1} t e^{-y_{1} t \beta} + \frac{1}{\lambda} e^{-\gamma t}. \nonumber
\end{align} 
The desired bound now follows easily, by choosing $\kappa = \min\left\{y_{1} \beta, \gamma\right\}$ and $k = K y_{1} + y_{1} + 1$.

\end{proof}

\subsection{Bounding Expansion}

\begin{theorem}\label{thmxy}  There exists $K_0$ (dependent only on $s_0,k$)
such that for {\em any} $T_0$ and any $\vec{x},\vec{y}\in D$
\beq\label{contracteq}
TV(\Psi_{c}^{s}(T_0,\vec{x}), \Psi_{c}^{s}(T_0,\vec{y})) \leq K_0\cdot TV(\vec{x},\vec{y}).
\eeq
\end{theorem}

\begin{remark} 
 As $K_0$ may be large, Theorem \ref{thmxy}, by itself,
does not give a contracting mapping.  It does limit how expanding $\Psi_{c}^{s}(T_0,\cdot)$
can be.
\end{remark}

\begin{remark} 
Let  $t$ be the number of nodes of $T_0$ at depth $s_0$.
Let $TV(\vec{x},\vec{y})=\eps$.  The expected number of nodes in $Ext(T_0)$
at level $s=s_{0}+D_{0}$ is then $tK_1$ with $K_1 = c^{D_{0}}$.  
The methods of Theorem \ref{n} would then give Theorem \ref{thmxy} with $K=K_1 t$.
However, when $c>1$ this $K$ would be unbounded in $t$.  Our concern is then
with large $t$ though, technically, the proof below works for all $t$.\\
\end{remark}


\begin{proof}
Let $t$ be the number of nodes of $T_0$ at depth $s_0$.
Let $TV(\vec{x},\vec{y})=\eps$.  We again couple $\vec{x},\vec{y}$.
Let $Y$ be the number of nodes in
$Ext(T_0)$ at level $s$.  Given $Y=y$, let us name these vertices $u_{1}, \ldots u_{y}$. Again we create two pictures. In picture $1$, we assign, mutually independently, labels $X_{i} \in \Sigma$ to $u_{i}$, with $X_{i} \sim \vec{x}$, and in picture $2$, label $Z_{i} \sim \vec{y}$. $(X_{i}, Z_{i}), 1 \leq i \leq y$ mutually independent, but $X_{i}, Z_{i}$ are coupled so that 
\begin{equation}
P[X_{i} \neq Z_{i}] = TV(\vec{x}, \vec{y}) = \epsilon.
\end{equation}
The probability of the event that for at least one $i$, $X_{i} \neq Z_{i}$ is then bounded above by $y \cdot \epsilon$.
\par Suppose $X \in \Sigma$ is the label of the root of $Ext(T_{0})$ in picture $1$, and $Z$ that in picture $2$, determined by using the recursion function $\Gamma$ repeatedly upwards starting at level $s$. Then $X \sim \Psi_{c}^{s}(T_{0}, \vec{x}), \ Z \sim \Psi_{c}^{s}(T_{0}, \vec{y})$. 
\par Recall, from Definition \ref{defbad}, that $BAD[T_{0}]$ is the event that $Ext(T_{0})$ is not $s_{0}$-universal. If $GOOD[T_{0}] = BAD[T_{0}]^{c}$, then under $GOOD[T_{0}]$, the Ehrenfeucht value of $Ext(T_{0})$ is completely determined by $Ext(T_{0})|_{s_{0}}$, which is $T_{0}$ itself. And as $T_{0}$ is fixed, this means that $EV[Ext(T_{0})]$ is then independent of $\vec{x}, \vec{y}$. Thus
\begin{align}
TV(\Psi_{c}^{s}(T_{0}, \vec{x}), \Psi_{c}^{s}(T_{0}, \vec{y})) \leq & P[X \neq Z] \nonumber\\
\leq & \sum_{y=0}^{\infty} P[Y = y] \cdot y \cdot \epsilon \cdot \mathbf{1}_{BAD[T_{0}]} \nonumber\\
=& E[Y \ \mathbf{1}_{BAD[T_{0}]}] \epsilon. \nonumber
\end{align}
From Theorem \ref{thmuniv} and Lemma \ref{lem2},
\beq\label{qq2}
TV(\Psi_{c}^{s}(T_0,\vec{x}),\Psi_{c}^{s}(T_0,\vec{y})) 
\leq A(t) \eps  \mbox{ with } A(t) = kte^{-t\kappa}.  
\eeq
Here $A=A(t)$ approaches zero as $t\ra\infty$ and so there exists $K_0$ such that $A\leq K_0$ for {\em any} choice of $t$.
\end{proof}

\subsection{Proving Contraction}

We first show Theorem \ref{thmcontract} in terms of the $TV$ metric.  
Pick $s_0$ sufficiently large
so that, say, the probability that $T|_{s_{0}}$ is not $s_0$-universal is
at most  $(2K_0)^{-1}$, $K_0$ given by Theorem \ref{thmxy}. This can be done because of Theorem \ref{rapidsuniv}. Let
$\vec{x},\vec{y}\in D$ with $\eps=TV(\vec{x},\vec{y})$.
We bound $TV(\Psi_{c}^s(\vec{x}),\Psi_{c}^s(\vec{y}))$ by Theorem \ref{thmtwostage}.
Consider $TV(\Psi_{c}^s(T_0,\vec{x}),\Psi_{c}^s(T_0,\vec{y}))$. 
When $T_0$ is $s_0$-universal this has value zero.  Otherwise its value
is bounded by $K_0\eps$ by Theorem \ref{thmxy}.  Theorem \ref{thmtwostage}
then gives
\beq\label{zz5} 
TV(\Psi_{c}^s(\vec{x}),\Psi_{c}^s(\vec{y})) \leq \frac{1}{2K_0}K_0\eps 
\leq \frac{\eps}{2}.
\eeq

Finally, we switch to the $L^2$ metric $\rho$.
For $B$ a sufficiently large constant the inequaliteis (\ref{m4})
yield, say,
\beq\label{qq5}
\rho(\Psi_{c}^{sB}(\vec{x}),\Psi_{c}^{sB}(\vec{y}))\leq \frac{1}{2}\rho(\vec{x},\vec{y}). 
\eeq
Then Theorem \ref{thmcontract} is satisfied with $s$ replaced by $sB$ and $\ah=\frac{1}{2}$.

\section{Implicit Function}\label{sectimplicit}

Here we deduce Theorem \ref{e} and hence Theorem \ref{a}, that $\Pr[A]$ is
always a $C^{\infty}$ function of $c$.  This follows from three results:

\ben
\item\label{proba} The function $\Del(c,\vec{x})=\Psi_c(\vec{x})$ has all derivatives of
all orders.
\item\label{propb} For each $c > 0$ the function
\beq\label{zz1}  F(\vec{x}) = \Psi_c(\vec{x})-\vec{x} \eeq
has a unique zero $\vec{x}=\vec{x}(c)$.
\item\label{propc} The function $\Psi_c:D\ra D$ is contracting in the sense
of Theorem \ref{thmcontract}.
\een

Let $A$ be the Jacobian of $\Psi_c$ at $\vec{x}(c)$.   From Property \ref{propc}
all of the eigenvalues of $A$ lie inside the complex unit circle.  Then $A=I$
is the Jacobian of $F$ from Property \ref{propb}. Then $(A-I)^{-1} = -\sum_{u=0}^{\infty}A^u$
is a convergent sequence, and so $A-I$ is invertible,  As by Property \ref{proba} the
function $\Del$ is smooth, the Implicit Function Theorem gives that the fixed point
function $\vec{x}(c)$ of $F$ is $C^{\infty}$.

\end{document}